\documentclass[12pt,a4paper]{article}

\usepackage{amssymb, amsmath, amsthm}

\usepackage{algorithm}
\usepackage{algpseudocode}

\usepackage[cm]{fullpage}
\usepackage[english]{babel}
\usepackage[pdftex]{graphicx}
\usepackage{booktabs}
\usepackage{xcolor}

\usepackage{algorithm}

\usepackage{enumitem}

\usepackage{hyperref}
\usepackage{autonum}

\newtheorem{thm}{Theorem}
\newtheorem{lem}{Lemma}
\newtheorem{prop}{Proposition}
\newtheorem{rem}{Remark}

\title{Barenblatt solutions for the time-fractional porous medium equation: approach via integral equations}
\author{
Josefa Caballero\thanks{Departamento de Matem\'aticas, Universidad de Las Palmas de Gran Canaria, Campus de Tafira Baja, $35017$ Las Palmas de Gran Canaria, Spain.}, \and
Hanna Okrasi{\'n}ska-P{\l}ociniczak\thanks{Department of Applied Mathematics, Wroclaw University of Environmental and Life Sciences, ul. C.K. Norwida 25, 50-275 Wroclaw, Poland}, \and 
{\L}ukasz P{\l}ociniczak\thanks{Faculty of Pure and Applied Mathematics, Wroclaw University of Science and Technology, Wyb. Wyspia\'nskiego 27, 50-370 Wroc{\l}aw, Poland, \underline{corresponding author:} \texttt{lukasz.plociniczak@pwr.edu.pl}}, \vspace{4pt}\and
Kishin  Sadarangani$^{\ast}$ 
}
\date{}

\begin{document}
\maketitle

\begin{abstract}
	This paper explores Barenblatt solutions of the time-fractional porous medium equation, characterized by a Caputo-type time derivative. Employing an integral equation approach, we rigorously prove the existence of these solutions and establish several fundamental properties, including upper and lower estimates, mass conservation, regularity, and monotonicity. To bridge theory and practice, we introduce a family of convergent numerical schemes specifically designed to compute the Barenblatt solutions, ensuring reliable and efficient approximations. The theoretical framework is enriched with various examples that illustrate the concepts and validate the effectiveness of the proposed numerical methods, enhancing the understanding and applicability of our results.
\end{abstract}


\section{Introduction}
The porous medium equation (PME) stands as a fundamental model in the analysis of nonlinear diffusion processes, with a legacy that goes back to Barenblatt seminal work in 1952 \cite{barenblatt11some} (see also \cite{zel1950towards}). Originally formulated to describe the flow of a gas through a porous medium, the PME captures the essence of nonlinear diffusion through a density-dependent diffusion coefficient that can vanish for low concentrations. This nonlinearity manifests itself in striking features such as the finite speed of propagation and the emergence of sharp interfaces, phenomena that distinguish it from classical linear diffusion models. These characteristics have made the PME a vital tool in various fields, including hydrology \cite{bear2013dynamics}, fluid dynamics \cite{oron1997long}, heat transfer \cite{nield2017heat}, and biology \cite{sherratt1990models}. A detailed account of the classical PME can be found in the monograph \cite{vazquez2007porous}. 

The advent of fractional calculus has ushered in a transformative extension of such models, enabling the description of anomalous diffusion processes where traditional integer-order derivatives fall short. The time-fractional porous medium equation integrates a fractional time derivative into the classical framework, accommodating memory effects and long-range interactions inherent in complex systems \cite{metzler2000random,meerschaert2019stochastic}. This generalization is particularly pertinent to applications in viscoelasticity \cite{mainardi2022fractional}, hydrology \cite{plociniczak2015analytical, plociniczak2014approximation, cushman2000fractional, benson2000application}, heat transfer \cite{povstenko2015fractional} and financial mathematics \cite{winkel2005electronic, scalas2000fractional} (for additional examples, see \cite{sun2018new}). In this work, we consider a nonlinear subdiffusion equation with fractional derivative, the time-fractional PME (see \cite{plociniczak2015analytical}). It has an interesting application in the description of various experiments measuring the imbibition properties of many new materials \cite{de2006water, de2006concentration, el2004neutron, ramos2008anomalous}. 

The following equation models moisture diffusion in porous media in which we account for the memory of the process \cite{li2019theory,plociniczak2015analytical}
\begin{equation}\label{eqn:MainEq0}
	\partial^\alpha_t u = \left(u^m u_x\right)_x, \quad x \in\mathbb{R}, \quad t > 0, \quad 0< \alpha < 1, \quad m > 0,
\end{equation}
where the diffusivity is chosen as a power-law which is a typical model of porous media. The nonlocal operator, the Caputo derivative, is defined on suitably chosen function spaces by (for more details see \cite{kubica2020time})
\begin{equation}\label{eqn:Caputo}
	\partial^\alpha_t u(x,t) = \frac{1}{\Gamma(1-\alpha)}\int_0^t (t-\tau)^{-\alpha} u_t(x,\tau)d\tau, \quad 0 < \alpha < 1.
\end{equation}
We note that, with a suitable integration by parts the above can be defined for functions of only $\beta$-H\"older regularity with $\beta > \alpha$ (see \cite{del2025note}). In addition, the Caputo derivative can be defined as $\partial_t^\alpha u(x,t) = I^{1-\alpha}_t u_t(x,t)$, where the fractional integral is defined by
\begin{equation}\label{eqn:FracInt}
	I^{\beta}_t u(x,t) = \frac{1}{\Gamma(\beta)}\int_0^t (t-\tau)^{\beta-1} u(x,\tau) d\tau, \quad \beta > 0.
\end{equation}
Frequently, we use the term \textit{classical solution} as a solution of any PDE with $\alpha = 1$.

Now, we would like to study an important problem of sudden mass release at $x = 0$. This is one of the most important early examples of problems for the classical nonlinear diffusion which possesses an exact solution discovered by Barenblatt \cite{barenblatt11some}, Zel'dovich and Kompaneets \cite{zel1950towards}, and shortly later by Pattle \cite{pattle1959diffusion}. This, in turn, gave rise to vigorous research and a number of important results on nonlinear diffusion in later years, from which we mention notable works by Aronson \cite{aronson1969regularity}, Gilding and Peletier \cite{gilding1976class}, Goncerzewicz \cite{gilding2000localization}, and the more application-oriented approach of Okrasi\'nski \cite{okrasinski1993approximate}. A very thorough account of the history of PME can be found in \cite{vazquez2007porous}. We assume that the total mass is conserved and the medium is initially dry
\begin{equation}\label{eqn:MassConservation0}
	\int_{-\infty}^\infty u(x,t) dx = 1, \quad u(x, 0) = 0 \text{ a.e.}
\end{equation} 
Note that this basically means that $u(x,0) = \delta(x)$ with $\delta$ being the Dirac distribution. For simplicity, the solution of \eqref{eqn:MainEq0} with the above condition will be called the \textit{Barenblatt solution} in the sequel. This function is very useful for inquiring properties of solutions of the PME in the general case. The striking feature of the Barenblatt solution is its compact support (which can be explicitly seen from the exact formula in \eqref{eqn:BarenblattExact}). This, in turn, implies a finite speed of wave propagation in the modeled process - a feature much needed in real-world scenarios. 

In more recent times, the PME has been generalized in various ways to account for nonlocal phenomena associated with different kinds of exotic media. As mentioned above, long-range forces can be modeled by the use of fractional Laplacian or similar nonlocal operators acting on the space variables. The fractional PME appeared in \cite{de2011fractional} in the specific case and in \cite{de2012general} in general, where the authors studied the existence, uniqueness, and regularity of the solutions. A different approach to a similar problem has been taken in \cite{caffarelli2011nonlinear, caffarelli2010asymptotic} roughly at the same time. The case of a bounded domain, which in the nonlocal space setting is not based only on a simple truncation, was treated in \cite{bonforte2015existence}. Moreover, the Barenblatt solution in the space-nonlocal case has been investigated in \cite{biler2011barenblatt,biler2015nonlocal} and \cite{huang2014explicit} in which even closed-form explicit solutions have been found. Some further research on compactly supported solutions and the finite speed of wave propagation has been carried out in \cite{stan2016finite}.

On the other hand, the memory effects of the corresponding physical diffusive process can be described by the time-fractional Caputo derivative \cite{metzler2000random, plociniczak2015analytical}. There are many works in the literature on the linear version of \eqref{eqn:MainEq0}, and the reader is invited to consult \cite{sakamoto2011initial, kubica2020time, kubica2018initial} for details. For the quasilinear case, the theory has been actively developed over the last few years. For example, strong global solvability has been obtained in \cite{zacher2012global, wittbold2021bounded}, while optimal decay estimates have been found in \cite{vergara2015optimal}. A general abstract approach based on fractional gradient flows in Hilbert spaces has been developed in \cite{akagi2019fractional} (some very recent developments can be found in \cite{akagi2025time}). Self-similar solutions of \eqref{eqn:MainEq0} in half-line (what models is the typical experimental setup for measuring the properties of the medium under moisture imbibition) have been investigated by one of the authors and his collaborators in \cite{plociniczak2014approximation} (approximate solutions),  \cite{plociniczak2018existence, lopez2024time} (existence and uniqueness) and \cite{plociniczak2019numerical, plociniczak2023linear} (numerical methods). Recently, a new line of investigations has been followed in which mathematicians are interested in the most general equations with both space- and time-nonlocal operators. For example, in \cite{allen2016parabolic} the regularity of a linear, but doubly nonlocal, parabolic equation has been investigated. These results were generalized in \cite{allen2017porous} for the nonlocal PME. In \cite{dipierro2019decay} optimal decay estimates have been found for a wide class of nonlinear problems. An approach based on viscosity solutions has been carried over in \cite{topp2017existence} and we also mention a recent work in which the existence and regularity of a general nonlocal diffusion equation are determined by the approach through numerical schemes \cite{del2023numerical}. Also, very recently, a version of doubly nonlocal PME has been considered in \cite{bonforte2024time}, where the authors proved the existence, uniqueness, regularity, comparison principle, and decay rates. Moreover, an interesting account of singular solutions has been given in \cite{chan2024singular}. 

The literature on numerical methods for solving differential equations with nonlocal operators is vast and we do not intend to give a full account. The reader is invited to consult \cite{li2019theory, li2015numerical} for a detailed exposition of the numerical methods used to approximate Caputo fractional derivatives (and other). Our approach in this paper is slightly different and is based on the methods used to solve integral equations (see \cite{linz1985analytical}).  

To our knowledge, the point-source problem for \eqref{eqn:MainEq0} was not studied in the literature in the non-linear case of $m>0$. For $m=0$, Barenblatt profile can be called the fundamental solution and can be explicitly given in terms of Wright's function \cite{gorenflo2000wright} (in the classical case it is, of course, the Gaussian). In the presence of spatial nonlocality, in this case modeled by the fractional Laplacian, similar explicit forms can be obtained with the use of other special functions \cite{mainardi2001fundamental, luchko1998scale}. These fundamental solutions do not have a compact support. In the quasilinear case \eqref{eqn:MainEq0} the solution is much less investigated, but we would like to mention a very interesting approach of providing probabilistic interpretation of the Barenblatt solutions given in \cite{de2020alternative}. As the classical example shows, these functions approach the Dirac delta as $t\rightarrow 0^+$. That is to say, the main PDE should be equipped with the measure initial datum and its solution should be interpreted in appropriate weak (or very weak) sense. We refer the reader to \cite{vazquez2007porous} for a thorough discussion of that topic in the classical case. However, the theory of the time-fractional version of a PME is still under very vigorous development, and in this paper, we refrain from moving in that direction because it would lead us too far from the main aim of our work. Some relevant works include \cite{zacher2012global}, where the authors consider the existence, uniqueness, and regularity of various versions of \eqref{eqn:MainEq0} with the initial datum being a function. 

The basic aim of this paper is to construct and investigate Barenblatt solutions of \eqref{eqn:MainEq0}. We show that this equation, analogously to the classical case, possesses a similarity solution which satisfies the mass conservation condition \eqref{eqn:MassConservation0}. This particular solution is then described by a nonlinear Volterra equation with a known kernel. Our analysis is then based on investigating various properties of this equation. We are able to find many of its features, such as regularity, monotonicity, symmetry, bounds, uniqueness (in certain sense), and asymptotic behavior. Moreover, for practical uses, we devise an efficient numerical scheme for computing the Barenblatt solution and prove its convergence to the positive solution. This is very important, since due to non-Lipschitzian nonlinearity, our equation always has a trivial solution. Our theory is then illustrated with various numerical examples that verify our findings. 

In the next section, we derive the main Volterra integral equation and thoroughly investigate it which leads to the construction of the Barenblatt solution of \eqref{eqn:MainEq0}. Then, in Section 3 we devise a numerical scheme based on the product integration method and prove its convergence.

\section{Barenblatt solution}
In this section, we first reduce our main problem to solving a nonlinear Volterra integral equation, then characterize its solution to finally build a Barenblatt solution of \eqref{eqn:MainEq0}. 

\subsection{Derivation of the integral equation}
Suppose that we look for solutions in the form
\begin{equation}\label{eqn:SelfSimilarSol}
	u(x,t) = t^{-a} U(z), \quad z := x t^{-b} \text{ for some } a,b > 0. 
\end{equation}
After substituting this form into the mass conservation \eqref{eqn:MassConservation0} we obtain
\begin{equation}
	1 = \int_{-\infty}^\infty u(x,t) dx = t^{-a} \int_{-\infty}^\infty U(x t^{-b}) dx = t^{b-a} \int_{-\infty}^\infty U(z) dz,
\end{equation}
where we change the variable into the self-similar. The above can only be satisfied when
\begin{equation}
	a = b. 
\end{equation}
Now, the space derivative transforms as
\begin{equation}\label{eqn:SpaceDerivative}
	\frac{\partial}{\partial x} = \frac{\partial z}{\partial x} \frac{\partial}{\partial z} = t^{-b} \frac{\partial}{\partial z}.
\end{equation}
Finally, we have to translate the Caputo derivative to a self-similar setting. For computational reasons, it is worth to note that for the vanishing initial condition as in \eqref{eqn:MassConservation0}, the Caputo derivative is equivalent to the Riemann-Liouville version defined by
\begin{equation}
	^{RL}\partial^\alpha_t u(x,t) = \frac{1}{\Gamma(1-\alpha)} \frac{\partial}{\partial t} \int_0^t (t-\tau)^{-\alpha} u(x,\tau) d\tau, \quad 0 < \alpha < 1.
\end{equation}
Therefore, with \eqref{eqn:SelfSimilarSol} and a change of variable $\tau = t \sigma$ we have
\begin{equation}
	\begin{split}
		^{RL}\partial^\alpha_t u(x,t) &= \frac{1}{\Gamma(1-\alpha)}\frac{\partial}{\partial t} \left(\int_0^t (t-\tau)^{-\alpha} \tau^{-a}U(x \tau^{-a}) d\tau\right) \\
		&= \frac{1}{\Gamma(1-\alpha)}\frac{\partial}{\partial t} \left(t^{1-a-\alpha}\int_0^1 (1-\sigma)^{-\alpha} \sigma^{-a}U(x t^{-a} \sigma^{-a}) d\sigma\right).
	\end{split}
\end{equation}
Now, by taking the derivative we obtain
\begin{equation}
	^{RL}\partial^\alpha_t u(x,t) = \frac{1}{\Gamma(1-\alpha)} \left[(1-a-\alpha) t^{-a-\alpha} + t^{1-a-\alpha} \frac{\partial}{\partial t}\right]\left(\int_0^1 (1-\sigma)^{-\alpha} \sigma^{-a}U(x t^{-a} \sigma^{-a}) d\sigma\right).
\end{equation}
But since $z = x t^{-a}$ we can use the chain rule to obtain
\begin{equation}
	\frac{\partial}{\partial t} = \frac{\partial z}{\partial t} \frac{\partial}{\partial z} = -a x t^{-a-1}\frac{\partial}{\partial z} = -a t^{-1} z \frac{\partial}{\partial z},  
\end{equation}
and, hence,
\begin{equation}\label{eqn:RLSS}
	^{RL}\partial^\alpha_t u(x,t) = \frac{t^{-a-\alpha}}{\Gamma(1-\alpha)} \left[(1-a-\alpha) - a z \frac{\partial}{\partial z}\right]\left(\int_0^1 (1-\sigma)^{-\alpha} \sigma^{-a}U(z \sigma^{-a}) d\sigma\right).
\end{equation}
Finally, using \eqref{eqn:SpaceDerivative} in our equation \eqref{eqn:MainEq0} the time variable will cancel only if $-2a - ma - a= -a - \alpha$, that is 
\begin{equation}
	a = b = \frac{\alpha}{2+m},
\end{equation} 
and an equation for $U$
\begin{equation}\label{eqn:SelfSimilarEq}
	\left(A - B z \frac{d}{dz}\right)F_{\alpha,m} U = \left(U^m U'\right)', \quad A = 1 - \alpha - \frac{\alpha}{2+m}, \quad B = \frac{\alpha}{2+m},
\end{equation}
with primes denoting the $z$-derivative and the \textit{Erd\'elyi-Kober}-type operator $F_{\alpha,m}$ is defined by
\begin{equation}\label{eqn:EK}
	F_{\alpha,m} U(z) = \frac{1}{\Gamma(1-\alpha)} \int_0^1 (1-\sigma)^{-\alpha} \sigma^{-\frac{\alpha}{2+m}} U(z \sigma^{-\frac{\alpha}{2+m}}) d\sigma.
\end{equation}
In addition, we have the mass conservation
\begin{equation}\label{eqn:MassConservation}
	\int_{-\infty}^\infty U(z) dz = 1.
\end{equation}
In the classical case, that is, when $\alpha \rightarrow 1^-$ such a problem has a solution, known as the Barenblatt solution (or Barenblatt-Zeldovich-Kompanyeets), which can be expressed explicitly as
\begin{equation}\label{eqn:BarenblattExact}
	U(z) = \left(D-\frac{z^2}{2(2+m)}\right)_+^\frac{1}{m}, \quad f_+(z) := \max\{f(z), 0\},
\end{equation}
where $D>0$ is chosen to satisfy the mass conservation \eqref{eqn:MassConservation} to be
\begin{equation}
	D = \left(\frac{\Gamma \left(\frac{3}{2}+\frac{1}{m}\right)}{\sqrt{2\pi(m+2)} \Gamma \left(1+\frac{1}{m}\right)}\right)^{\frac{2 m}{m+2}}.
\end{equation} 
Note that this is a compactly supported symmetric solution on $[-\sqrt{2D(2+m)},-\sqrt{2D(2+m)}]$, which is in striking contrast to the linear case $m=0$, where the solution is a Gaussian. Furthermore, by a straightforward calculation it follows that the Barenblatt solution is a continuous function $m^{-1}-$ H\"older. No such type of explicit solutions are known for the case $0<\alpha<1$. 

We would like to investigate the existence of Barenblatt solutions for \eqref{eqn:SelfSimilarEq} by using integral equations. It is natural to expect that for $0<\alpha<1$ the support of our solution is also compact, that is, there exists $z_0 > 0$ such that
\begin{equation}\label{eqn:CompactSupport}
	U(z) = 0 \quad \text{for} \quad |z| \geq z_0.
\end{equation}
Moreover, as is evident from \eqref{eqn:SelfSimilarEq} by substituting $z \mapsto -z$ our solution is symmetric, and it is sufficient to consider only half of the domain $z\in [-z_0,0]$ with half of the mass present there
\begin{equation}\label{eqn:MassConservationHalf}
	\int_{-z_0}^0 U(z) dz = \frac{1}{2}.
\end{equation}
Also, it is natural in the physical sense to assume that there is no flux through the interface of the water and the medium (this is true for similar problems, see \cite{plociniczak2018existence}), that is
\begin{equation}
	-U^m U' = 0 \quad \text{for} \quad |z| = z_0.
\end{equation}
We can now integrate the self-similar equation from $-z_0$ to $z$ to obtain
\begin{equation}\label{eqn:MainEqSSInt}
	- B z F_{\alpha, m} U + (A + B) \int_{-z_0}^z F_{\alpha, m} U(y)dy = U^m U'. 
\end{equation}
Noticing that $U^m U' = (m+1)^{-1} (U^{m+1})'$ we can integrate once more to obtain
\begin{equation}
	-B \int_{-z_0}^z w F_{\alpha,m}U(w) dw + (A+B)\int_{-z_0}^z \int_{-z_0}^w F_{\alpha, m} U(y) dy dw = \frac{1}{m+1} U(z)^{m+1},
\end{equation}
or, by Fubini's Theorem and taking the root
\begin{equation}\label{eqn:FixedPoint1}
	U(z) = \left((m+1)\int_{-z_0}^z \left((A+B)(z-w) - Bw\right)F_{\alpha, m} U(w) dw\right)^{\frac{1}{m+1}}, \quad -z_0 \leq z \leq 0,
\end{equation}
which is a fixed-point form of the equation \eqref{eqn:SelfSimilarEq} (and is well-defined for $U \geq 0$ since $z-w \geq 0$, $-w \geq 0$, and the kernel of $F_{\alpha,m}$ is positive). The above can also be simplified by invoking the definition of the E-K operator \eqref{eqn:EK} and noticing that $U(z) = 0$ for $z\leq -z_0$. For $-z_0 \leq z \leq 0$, we have 
\begin{equation}
	\begin{split}
		\int_{-z_0}^z&\left((A+B)(z-w) - Bw\right)F_{\alpha, m} U(w) dw \\
		&= \frac{1}{\Gamma(1-\alpha)} \int_{-z_0}^z \left((A+B)(z-w) - Bw\right) \int_{\left(-\frac{w}{z_0}\right)^{(2+m)/\alpha}}^1 (1-\sigma)^{-\alpha} \sigma^{-\frac{\alpha}{2+m}} U(w \sigma^{-\frac{\alpha}{2+m}}) d\sigma dw \\
		&= \frac{1}{\Gamma(1-\alpha)} \left(-\frac{2+m}{\alpha}\right) \int_{-z_0}^z \left((A+B)(z-w) - Bw\right) \int_{-z_0}^w \left(1-\left(\frac{w}{\tau}\right)^\frac{2+m}{\alpha}\right)^{-\alpha} w^{\frac{2+m}{\alpha}-1}\tau^{-\frac{2+m}{\alpha}}U(\tau) d\tau dw,
	\end{split}
\end{equation}
where we have used a natural change of the variable $\tau = w \sigma^{-\alpha/(2+m)}$. Now, by Fubini's Theorem, we have
\begin{equation}
	\begin{split}
		\int_{-z_0}^z&\left((A+B)(z-w) - Bw\right)F_{\alpha, m} U(w) dw \\
		&=-\frac{2+m}{\alpha\Gamma(1-\alpha)} \int_{-z_0}^z \left(\int_\tau^z \left((A+B)(z-w) - Bw\right) \left(1-\left(\frac{w}{\tau}\right)^\frac{2+m}{\alpha}\right)^{-\alpha} w^{\frac{2+m}{\alpha}-1} dw\right)\tau^{-\frac{2+m}{\alpha}} U(\tau) d\tau.
	\end{split}
\end{equation}
Finally, we can substitute the inner integral $\sigma = (w/\tau)^{(2+m)/\alpha}$ to obtain
\begin{equation}
	\begin{split}
		\int_{-z_0}^z&\left((A+B)(z-w) - Bw\right)F_{\alpha, m} U(w) dw \\
		&=\frac{1}{\Gamma(1-\alpha)} \int_{-z_0}^z \left( \int_{\left(\frac{z}{\tau}\right)^\frac{2+m}{\alpha}}^1\left((A+B)(z-\tau \sigma^\frac{\alpha}{2+m}) - B\tau \sigma^\frac{\alpha}{2+m}\right)(1-\sigma)^{-\alpha}d\sigma\right)U(\tau)d\tau.
	\end{split}
\end{equation}
It is convenient to define
\begin{equation}\label{eqn:KDefinition}
	K(z,\tau) := \int_{\left(\frac{z}{\tau}\right)^\frac{2+m}{\alpha}}^1\left((A+B)(z-\tau \sigma^\frac{\alpha}{2+m}) - B\tau \sigma^\frac{\alpha}{2+m}\right)(1-\sigma)^{-\alpha}d\sigma.
\end{equation}
Actually, the above defined kernel has an exact representation and a lot can be said about its behavior.
\begin{prop}\label{prop:Kernel}
	For any $\tau \leq z \leq 0$, $0<\alpha<1$, and $m > 0$ we have
	\begin{equation}\label{eqn:KExact}
		\begin{split}
			K(z,\tau) &= z \left(1-\left(\frac{z}{\tau}\right)^\frac{2+m}{\alpha}\right)^{1-\alpha} \\
			&- \left(1-\alpha+\frac{\alpha}{2+m}\right)\tau \left(\beta\left(1+\frac{\alpha}{2+m},1-\alpha\right)-\beta_{\left(\frac{z}{\tau}\right)^\frac{2+m}{\alpha}}\left(1+\frac{\alpha}{2+m},1-\alpha\right)\right),
		\end{split}
	\end{equation}
	where $\beta = \beta(x,y)$ is the Euler beta function, and $\beta_\xi = \beta_{\xi}(x,y) := \int_0^\xi t^{x-1} (1-t)^{y-1} dt$ its incomplete version. Moreover, we have the following:
	\begin{enumerate}[label=\roman*)]
		\item Positivity: $K(z,t) \geq 0$ with $K(z,z) = 0$.
		\item Boundedness:
		\begin{equation}
			0\leq K_-(z,\tau) \leq K(z,\tau) \leq K(0,\tau),
		\end{equation}
		where 
		\begin{equation}\label{eqn:KBound}
			\begin{split}
				&K_-(z,\tau) := -\frac{\alpha}{2+m} \beta\left(1+\frac{\alpha}{2+m},1-\alpha\right)\left( 1-\left(\frac{z}{\tau}\right)^\frac{2+m}{\alpha}\right)^{1-\alpha}\tau,\\  &K(0,\tau) =- \left(1-\alpha+\frac{\alpha}{2+m}\right)\beta\left(1+\frac{\alpha}{2+m},1-\alpha\right)\tau.
			\end{split}
		\end{equation}
		\item Global H\"older regularity in the first variable:
		\begin{equation}\label{eqn:KHolder}
			|K(z, \tau) - K(w, \tau)| \leq C |z-w|^{1-\alpha}, \quad \tau \leq \min\{z,w\}, \quad z, w \leq 0,
		\end{equation}
		for some $C>0$. 
		\item Behavior for $\tau \sim -\infty$
		\begin{equation}
			K(z,\tau) \sim z- \left(1-\alpha+\frac{\alpha}{2+m}\right)\beta\left(1+\frac{\alpha}{2+m},1-\alpha\right)\tau.
		\end{equation}
		\item Behavior for $\tau \sim z$
		\begin{equation}\label{eqn:KAsymz}
			K(z,\tau) \sim -\frac{\alpha}{(1-\alpha)(2+m)} z \left(\frac{2+m}{\alpha} \left(1-\frac{z}{\tau}\right)\right)^{1-\alpha}.
		\end{equation}
	\end{enumerate}
\end{prop}
\begin{proof}
	By rewriting the integral in \eqref{eqn:KDefinition} and recalling that according to \eqref{eqn:SelfSimilarEq} we have $A+B = 1-\alpha$ and $A + 2B = 1-\alpha + \alpha/(2+m)$, we have
	\begin{equation}\label{eqn:KDefinition2}
		K(z,\tau) = (1-\alpha) z \int_{\left(\frac{z}{\tau}\right)^\frac{2+m}{\alpha}}^1 (1-\sigma)^{-\alpha} d\sigma - \left(1-\alpha+\frac{\alpha}{2+m}\right) \tau \int_{\left(\frac{z}{\tau}\right)^\frac{2+m}{\alpha}}^1 \sigma^{\frac{\alpha}{2+m}} (1-\sigma)^{-\alpha}d\sigma.
	\end{equation}
	We immediately notice that the first integral is elementary, while the second is related to beta and incomplete beta functions with parameters $1+\alpha/(2+m)$ and $1-\alpha$.  
	
	Now, since in the integral \eqref{eqn:KDefinition} we have $-\tau \sigma^{\alpha/(2+m)} \geq -z \geq 0$, it follows that $z -\tau \sigma^{\alpha/(2+m)} \geq 0$. Therefore, both terms of the integrand are non-negative, and thus $K(z,\tau) \geq 0$. It is clear that when $\tau \rightarrow z^-$ the integration range vanishes yielding $K(z,z) = 0$. 
	
	The kernel bounds can be found from \eqref{eqn:KDefinition}. By above estimates we have $z-\tau \sigma^{\alpha/(2+m)}\geq 0$ and hence, 
	\begin{equation}
		K(z,\tau) \geq - \frac{\alpha}{2+m} \tau \int_{\left(\frac{z}{\tau}\right)^\frac{2+m}{\alpha}}^1 \sigma^{\frac{\alpha}{2+m}} (1-\sigma)^{-\alpha}d\sigma.
	\end{equation}
	Next, we want to bound the integral from below by an elementary function. To this end, substitute $\sigma = s + (1-s)(z/\tau)^{(2+m)/\alpha}$ to fix the limits of integration. After elementary manipulations, the integral then becomes
	\begin{equation}
		\int_{\left(\frac{z}{\tau}\right)^\frac{2+m}{\alpha}}^1 \sigma^{\frac{\alpha}{2+m}} (1-\sigma)^{-\alpha}d\sigma = \left(1-\left(\frac{z}{\tau}\right)^\frac{2+m}{\alpha}\right)^{1-\alpha} \int_0^1 \left(s + (1-s)\left(\frac{z}{\tau}\right)^{\frac{2+m}{\alpha}}\right)^\frac{\alpha}{2+m} (1-s)^{-\alpha} ds.
	\end{equation}
	But, it is clear that $ s + (1-s)(z/\tau)^{(2+m)/\alpha} \geq s$, hence
	\begin{equation}
		K(z,\tau) \geq - \frac{\alpha}{2+m} \tau \left(1-\left(\frac{z}{\tau}\right)^\frac{2+m}{\alpha}\right)^{1-\alpha} \int_0^1 s^\frac{\alpha}{2+m} (1-s)^{-\alpha} ds,
	\end{equation}
	which is precisely the definition of the beta function. Hence, the lower bound \eqref{eqn:KBound} is proved. On the other hand, the upper bound can be found directly from \eqref{eqn:KExact} by noticing that $z \leq 0$ and the incomplete beta function is positive. Similarly, the asymptotic behavior for $\tau \rightarrow -\infty$ follows from the fact that $z/\tau \rightarrow 0$. On the other hand, when $\tau \rightarrow z^-$ we have $(z/\tau)^{(2+m)/\alpha}\rightarrow 1$. By standard results from asymptotic theory, from \eqref{eqn:KDefinition} we further have
	\begin{equation}
		\begin{split}
			K(z,\tau) 
			&\sim \left((1-\alpha) (z-\tau) - \frac{\alpha}{2+m} \tau\right) \int_{\left(\frac{z}{\tau}\right)^\frac{2+m}{\alpha}}^1 (1-\sigma)^{-\alpha} d\sigma \\
			&\sim - \frac{\alpha}{(1-\alpha)(2+m)} \tau \left(1-\left(\frac{z}{\tau}\right)^\frac{2+m}{\alpha}\right)^{1-\alpha} \sim 
			- \frac{\alpha}{(1-\alpha)(2+m)} \tau \left(\frac{2+m}{\alpha}\left(1-\frac{z}{\tau}\right)\right)^{1-\alpha},
		\end{split}
	\end{equation}
	what concludes the proof of the asymptotic behavior. 
	
	Finally, the H\"older regularity of the kernel follows from the observation that from \eqref{eqn:KExact} for fixed $\tau \leq \tau_0 < z$, that is, away from $z \rightarrow \tau^+$, the function $K(\cdot, \tau)$ is smooth as a sum of an elementary function and an integral. On the other hand, for $z \rightarrow \tau^{+}$, the asymptotic relation \eqref{eqn:KAsymz} proved shows that $K(z, \tau) = O((z-\tau)^{1-\alpha})$. Therefore, we conclude that $K(\cdot,\tau)$ is a $(1-\alpha)-$H\"older function.
\end{proof}
\begin{rem}
	By the asymptotic relations, we see that the kernel $K(z,\tau)$ behaves linearly for $\tau$ away from $z$ and according to the power law $(1-z/\tau)^{1-\alpha}$ for $\tau$ close to $z$. In Fig. \ref{fig:Kernel} we present an illustrative example.
\end{rem}

\begin{figure}
	\centering
	\includegraphics{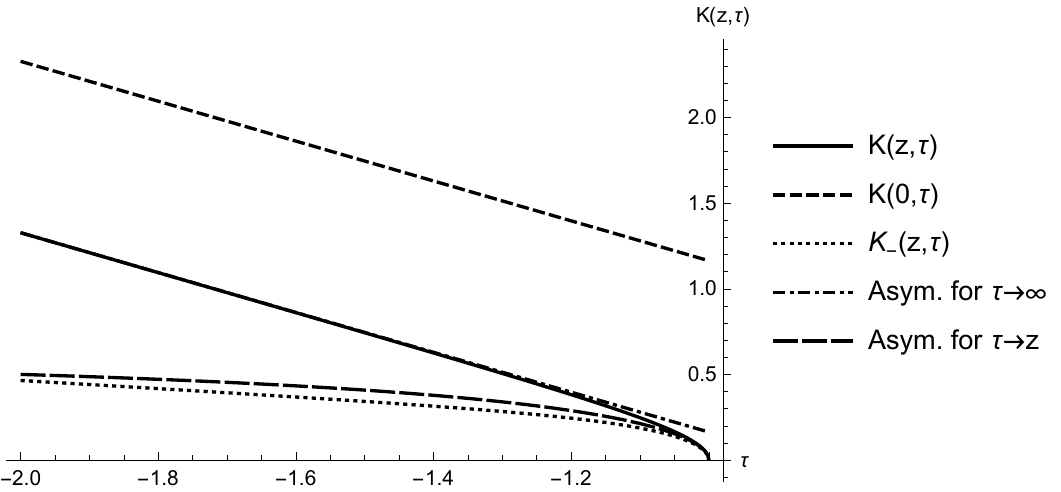}
	\label{fig:Kernel}
	\caption{Exemplary plot of the kernel \eqref{eqn:KDefinition} along its approximation for $z= - 1$, $\alpha = 0.5$ and $m=2$. }
\end{figure}

In this way, our equation \eqref{eqn:FixedPoint1} transforms into
\begin{equation}\label{eqn:FixedPoint2}
	U(z) = \left(\frac{m+1}{\Gamma(1-\alpha)}\int_{-z_0}^z K(z,\tau)U(\tau) d\tau\right)^\frac{1}{m+1}, \quad -z_0 \leq z \leq 0.
\end{equation}
Immediately, we can see that by Proposition \ref{prop:Kernel} $i)$ the above can only have non-negative real solutions or a \textit{trivial solution} that is identically equal to zero. 

\subsection{The solution of the integral equation}
We start with the following existence result. 
\begin{prop}\label{prop:Existence}
	There exists a continuous solution to \eqref{eqn:FixedPoint2} with arbitrary $z_0 > 0$.
\end{prop}
\begin{proof}
	We will use Schauder's fixed point theorem (see for ex. \cite{zeidler2012applied}, Section 1.15) to prove existence of the solution. Define the following set
	\begin{equation}
		A:=\left\{V \in C([-z_0,0]): \, \|V\|_\infty \leq \left(\frac{m+1}{\Gamma(1-\alpha)}\int_{-z_0}^0 K(0, \tau) d\tau\right)^\frac{1}{m}\right\},
	\end{equation}
	and the operator $\mathcal{N}: C([-z_0,0]) \mapsto C([-z_0,0])$ by
	\begin{equation}\label{eqn:NOperator}
		\mathcal{N}(V) := \left(\frac{m+1}{\Gamma(1-\alpha)}\int_{-z_0}^z K(z,\tau)V(\tau) d\tau\right)^\frac{1}{m+1}. 
	\end{equation}
	It is clear that since $A$ is a ball, it is a closed convex set. Moreover, for any $V\in A$ we have
	\begin{equation}
		\begin{split}
			\mathcal{N}(V) 
			&\leq \left(\left(\frac{m+1}{\Gamma(1-\alpha)}\int_{-z_0}^0 K(0, \tau) d\tau\right)^\frac{1}{m} \frac{m+1}{\Gamma(1-\alpha)}\int_{-z_0}^0 K(z,\tau)d\tau\right)^\frac{1}{m+1} \\
			&\leq \left(\left(\frac{m+1}{\Gamma(1-\alpha)}\int_{-z_0}^0 K(0, \tau) d\tau\right)^\frac{1}{m} \frac{m+1}{\Gamma(1-\alpha)}\int_{-z_0}^0 K(0,\tau)d\tau\right)^\frac{1}{m+1},
		\end{split}
	\end{equation}
	where the second inequality follows from \eqref{eqn:KExact}. Now, by simple addition of exponents we have
	\begin{equation}\label{eqn:UBounded}
		\mathcal{N}(V) \leq \left(\left(\frac{m+1}{\Gamma(1-\alpha)}\int_{-z_0}^0 K(0, \tau) d\tau\right)^{1+\frac{1}{m}} \right)^\frac{1}{m+1} = \left(\frac{m+1}{\Gamma(1-\alpha)}\int_{-z_0}^0 K(0, \tau) d\tau\right)^\frac{1}{m},
	\end{equation}
	hence, the operator $\mathcal{N}$ maps $A$ into itself. To use Schauder's theorem, we only need to show that $\mathcal{N}(A)$ is relatively compact. To this end, take any sequence of functions $\{W_n\}_n \in \mathcal{N}(A)$. Then, for each $n\in\mathbb{N}$ there exists a function $V_n\in A$ such that $\mathcal{N}(V_n) = W_n$. We will show that $\mathcal{N}(V_n)$ has a convergent subsequence, which in turn is implied by the Arzel\`a-Ascoli theorem (see for ex. \cite{zeidler2012applied}, Section 1.11). By \eqref{eqn:UBounded}, we immediately see that $\{\mathcal{N}(V_n)\}_n$ is uniformly bounded. It remains to prove that this sequence is equicontinuous. Fix, $- z_0 < w < z < 0$ and arbitrarily $\epsilon > 0$. Now, use the reverse H\"older inequality $|a^p - b^p| \leq |a-b|^p$ with $0<p<1$ in \eqref{eqn:NOperator} to obtain
	\begin{equation}\label{eqn:Equi}
		\begin{split}
			\mathcal{N}(V_n)(z) &- \mathcal{N}(V_n)(w) 
			\leq \left(\frac{m+1}{\Gamma(1-\alpha)}\right)^\frac{1}{m+1} \left(\int_{-z_0}^z K(z,\tau)V_n(\tau) d\tau-\int_{-z_0}^w K(w,\tau)V_n(\tau) d\tau\right)^\frac{1}{m+1} \\
			&\leq \left(\frac{m+1}{\Gamma(1-\alpha)}\right)^\frac{1}{m+1} \left(\int_{w}^z K(z,\tau)V_n(\tau) d\tau+\int_{-z_0}^w \left(K(z,\tau) - K(w,\tau)\right)V_n(\tau) d\tau\right)^\frac{1}{m+1} \\
			&\leq \underbrace{\left(\frac{m+1}{\Gamma(1-\alpha)}\right)^\frac{1}{m}\left(\int_{-z_0}^0 K(0, \tau) d\tau\right)^\frac{1}{m+1}}_{:= C_0} \left(\int_{w}^z K(z,\tau) d\tau+\int_{-z_0}^w \left(K(z,\tau) - K(w,\tau)\right) d\tau\right)^\frac{1}{m+1},
		\end{split}
	\end{equation}
	where in the second inequality we have added and subtracted $\int_{-z_0}^w K(z,\tau)V_n(\tau) d\tau$ and in the third the fact that $V_n \in A$, and hence, is bounded. Now, note that by \eqref{eqn:KDefinition} or \eqref{eqn:KExact} the kernel $K(z,\tau)$ is a continuous function of two variables on $-z_0 \leq \tau \leq z$. Therefore, it is uniformly bounded by some constant $C_1>0$ (which can be found explicitly from \eqref{eqn:KExact}). Moreover, since its domain is compact, the kernel is actually uniformly continuous. Therefore, there exists $\delta_0 > 0$ such that for $|z - w| < \delta_0$, we have
	\begin{equation}
		\left|K(z,\tau) - K(w,\tau)\right| < \frac{\epsilon^{m+1}}{2(z_0+w)C_0^{m+1}}. 
	\end{equation}
	Therefore,
	\begin{equation}
		\mathcal{N}(V_n)(z) - \mathcal{N}(V_n)(w) 
		< C_0 \left(C_1(z-w) + \frac{1}{2C_0^{m+1}}\epsilon^{m+1} \right)^\frac{1}{m+1}.
	\end{equation}
	Now, set $\delta = \min\{\delta_0, \epsilon^{m+1}/(2C_0^{m+1}C_1)\}$ and observe that for $|z-w| < \delta$ we have
	\begin{equation}
		\mathcal{N}(V_n)(z) - \mathcal{N}(V_n)(w) 
		< C_0 \left(\frac{1}{2C_0^{m+1}}\epsilon^{m+1} + \frac{1}{2C_0^{m+1}}\epsilon^{m+1} \right)^\frac{1}{m+1} = \epsilon.
	\end{equation}
	Since $\delta$ does not depend on $n$, we have shown that the sequence $\{\mathcal{N}(V_n)\}_n$ is equicontinuous. From the Arzel\`a-Ascoli theorem it follows that there exists a subsequence $\{\mathcal{N}(V_{n_k})\}_{n_k}$ convergent to a function $W \in \mathcal{N}(A)$. The set $\mathcal{N}(A)$ is therefore relatively compact. Schauder's theorem then grants that $\mathcal{N}$ has a fixed point in $A$.
\end{proof}

Before we proceed to investigating the nontrivial solution of \eqref{eqn:FixedPoint2}, we prove the following auxiliary lemma, which is a variant of Gronwall's inequality. We include its simple proof for completeness. It states that a function satisfying a certain nonlinear inequality has to be separated from zero by a constant. 
\begin{lem}\label{lem:Gronwall}
	Let $y=y(t):\mathbb{R}^+\to \mathbb{R}^+$ be a function satisfying
	\begin{equation}
		y(t)^{m+1} \geq D \int_0^1 s^{\gamma} y(s t) ds, 
	\end{equation}
	for some constant $D > 0$, $m > 0$, and $\gamma > -1$. Then,
	\begin{equation}
		y(t) \geq \left(\frac{D}{1+\gamma}\right)^\frac{1}{m}.
	\end{equation}
\end{lem}
\begin{proof}
	We start by changing the variable $w = s t$, then
	\begin{equation}
		y(t)^{m+1} \geq D t^{-1-\gamma} \int_0^t w^{\gamma} y(w) dw.
	\end{equation}
	Define $u(t) := \int_0^t w^{\gamma} y(w) dw$ and differentiate to obtain
	\begin{equation}
		u'(t) = t^\gamma y(t) \geq t^\gamma\left(D t^{-1-\gamma} u(t)\right)^\frac{1}{m+1} = D^\frac{1}{m+1} t^{\gamma - \frac{1+\gamma}{m+1}} u(t)^\frac{1}{m+1},
	\end{equation}
	where we used our assumption. Now, this is an ordinary differential inequality that can be resolved by separating variables and identifying a derivative
	\begin{equation}
		\frac{m+1}{m}\left(u(t)^\frac{m}{m+1}\right)'\geq D^\frac{1}{m+1} t^{\gamma - \frac{1+\gamma}{m+1}}.
	\end{equation}
	Since by definition, $u(0) = 0$ we can integrate to obtain
	\begin{equation}
		u(t)^\frac{m}{m+1} \geq \frac{m}{m+1} \frac{D^\frac{1}{m+1} }{1+\gamma- \frac{1+\gamma}{m+1}} t^{1 + \gamma - \frac{1+\gamma}{m+1}} = \frac{D^\frac{1}{m+1}}{1 + \gamma} t^{\frac{m(1+\gamma)}{m+1}}.
	\end{equation}
	Hence, going back to the original function leads to
	\begin{equation}
		y(t)^{m+1} \geq D t^{-1-\gamma} D^\frac{1}{m} \frac{1}{\left(1+\gamma\right)^\frac{m+1}{m}} t^{1+\gamma} = \left(\frac{D}{1+\gamma}\right)^{\frac{m+1}{m}}. 
	\end{equation}
	Taking the root finishes the proof. 
\end{proof}

We are now ready to prove various properties, estimates, and regularity results.
\begin{thm}\label{thm:Solution}
	If $U=U(z)$ for $-z_0 \leq z \leq 0$ is a nontrivial solution of \eqref{eqn:FixedPoint2}, then the following hold:
	\begin{enumerate}[label=\roman*)]
		\item Monotonicity: $U=U(z)$ is increasing for $-z_0 < z \leq 0$. 
		\item Local differentiability: $U(z)$ is analytic away from $z_0$. In particular, the derivative at $z = 0$ is equal to 
		\begin{equation}\label{eqn:Derivative0}
			U'(0) = \frac{U(0)^{-m}}{\Gamma(1-\alpha)}\int_{-z_0}^0 U(\tau) d\tau.
		\end{equation}
		When $U$ satisfies the conservation of mass \eqref{eqn:MassConservationHalf} we have
		\begin{equation}
			U'(0) = \frac{U(0)^{-m}}{2\Gamma(1-\alpha)}.
		\end{equation}
		\item Global H\"older regularity:
		\begin{equation}
			|U(z)-U(w)|\leq C |z-w|^\frac{1-\alpha}{1+m}, \quad -z_0\leq z,w \leq 0,
		\end{equation}
		for some $C>0$.
		\item Bounds:
		\begin{equation}\label{eqn:UBound}
			U(z) \leq \left(\frac{m+1}{\Gamma(1-\alpha)}\int_{-z_0}^0 K(0, \tau) d\tau\right)^\frac{1}{m} = \left(\frac{m+1}{\Gamma(1-\alpha)}\left(1-\alpha+\frac{\alpha}{2+m}\right)\beta\left(1+\frac{\alpha}{2+m},1-\alpha\right)\frac{z_0^2}{2}\right)^\frac{1}{m},
		\end{equation}
		and 
		\begin{equation}
			U(z) \geq \left(\frac{\alpha(1+m)}{(2+m)\Gamma(1-\alpha)} \beta\left(1+\frac{\alpha}{2+m},1-\alpha\right) z_0^\alpha \right)^\frac{1}{m}(z_0+z)^\frac{2-\alpha}{m}.
		\end{equation}
		\item Behavior for $z \rightarrow -z_0^+$:
		\begin{equation}
			U(z) \sim \left(\frac{\frac{(m+1) z_0^\alpha}{\Gamma(2-\alpha)} \left(\frac{\alpha}{2+m} \right)^{\alpha}}{1+\frac{2-\alpha}{m}}\right)^\frac{1}{m} (z_0+z)^\frac{2-\alpha}{m}.
		\end{equation}
	\end{enumerate}
\end{thm}
\begin{proof}
	By the exact formula for the kernel $K=K(z,\tau)$ given in Proposition \ref{prop:Kernel} we observe that it is analytic on $-z_0 \leq \tau \leq \tau_0 < z \leq 0$ for any $-z_0<\tau_0 <z$. The H\"older continuity of $K(z,\tau)$ is only local near $\tau \rightarrow z^-$, which due to the limits of integration in \eqref{eqn:FixedPoint2} affects the solution $U(z)$ only when $z \rightarrow -z_0^+$ (see the concrete calculations in \eqref{eqn:KAsymCalcu} below). Hence, since the integrand in \eqref{eqn:FixedPoint2} is uniformly continuous in that region, we conclude that the solution $U(z)$ is differentiable. Continuing inductively, we conclude that in fact it is analytic on an interval $-z_0 < w \leq z \leq 0$. Moreover, by taking the $(m+1)$-th power of \eqref{eqn:FixedPoint2} we obtain
	\begin{equation}\label{eqn:UDerivative}
		U(z)^{m} U'(z) = \frac{1}{\Gamma(1-\alpha)}\int_{-z_0}^z K_z(z,\tau) U(\tau) d\tau,
	\end{equation}
	since by Proposition \ref{prop:Kernel} we have $K(z,z) = 0$. On the other hand, using the exact form of the kernel \eqref{eqn:KDefinition2} we can further obtain
	\begin{equation}
		\begin{split}
			K_z(z,\tau) 
			&= (1-\alpha) \int_{\left(\frac{z}{\tau}\right)^\frac{2+m}{\alpha}}^1 (1-\sigma)^{-\alpha} d\sigma + \frac{2+m}{\alpha} \left(\frac{z}{\tau}\right)^{\frac{2+m}{\alpha}-1}\frac{1}{\tau} \left(-(1-\alpha)z \left(1-\left(\frac{z}{\tau}\right)^{\frac{2+m}{\alpha}}\right)^{-\alpha} \right. \\
			&\left. + \left(1-\alpha+\frac{\alpha}{2+m}\right) \tau  \left(\frac{z}{\tau}\right) \left(1-\left(\frac{z}{\tau}\right)^\frac{2+m}{\alpha}\right)^{-\alpha}\right). \\
			&= (1-\alpha) \int_{\left(\frac{z}{\tau}\right)^\frac{2+m}{\alpha}}^1 (1-\sigma)^{-\alpha} d\sigma +\left(\frac{z}{\tau}\right)^{\frac{2+m}{\alpha}} \left(1-\left(\frac{z}{\tau}\right)^\frac{2+m}{\alpha}\right)^{-\alpha},
		\end{split}
	\end{equation}
	which is positive, hence, from \eqref{eqn:UDerivative} it follows that $U$ is increasing. Moreover, plugging $z=0$ we see that the second term vanishes leaving $K_z(0,\tau) = 1$. The formula \eqref{eqn:Derivative0} is then proved. 
	
	Now, we turn to global regularity. By exactly the same computations as in \eqref{eqn:Equi} we can show that
	\begin{equation}
		|U(z) - U(w)| \leq C_0 \left(\int_{w}^z K(z,\tau) d\tau+\int_{-z_0}^w \left|K(z,\tau) - K(w,\tau)\right| d\tau\right)^\frac{1}{m+1}.
	\end{equation}
	The first term above is obviously Lipschitz, while the second due to Proposition \ref{prop:Kernel} is $(1-\alpha)-$H\"older. Therefore, there exists a constant $C_1>0$ such that
	\begin{equation}
		|U(z) - U(w)| \leq C_1 |z-w|^{\frac{1-\alpha}{m+1}},
	\end{equation}
	therefore, the solution is globally H\"older continuous. 
	
	When $U$ and $K$ are positive, we have
	\begin{equation}
		U(z) \leq U(z)^\frac{1}{m+1} \left(\frac{m+1}{\Gamma(1-\alpha)}\int_{-z_0}^z K(z,\tau) d\tau\right)^\frac{1}{m+1}.
	\end{equation}
	Canceling the common term and using the bound by $K(0,\cdot)$ \eqref{eqn:KBound} yields \eqref{eqn:UBound}. The exact form of the integral of $K(0,\cdot)$ is straightforward to obtain as it is a linear function. The asymptotics for $z\rightarrow -z_0^{+}$ follows from the behavior of the kernel as proved in Proposition \ref{prop:Kernel}. To see this, let $z+z_0 =: h$ with $h\rightarrow 0^+$. Then, by Proposition \ref{prop:Kernel} we have
	\begin{equation}\label{eqn:KAsymCalcu}
		\begin{split}
			U(-z_0 + h) 
			&= \left(\frac{m+1}{\Gamma(1-\alpha)}\int_{-z_0}^{-z_0+h} K(-z_0 +h, \tau) U(\tau) d\tau\right)^\frac{1}{m+1} \\
			&\sim \left(U(-z_0+h)\frac{m+1}{\Gamma(1-\alpha)} \frac{\alpha (-z_0+h)}{(1-\alpha)(2+m)} \left(\frac{2+m}{\alpha} \right)^{1-\alpha} \int_{-z_0}^{-z_0+h}  \left(1-\frac{-z_0+h}{\tau}\right)^{1-\alpha}d\tau \right)^\frac{1}{m+1} \\
			&\sim \left(U(-z_0+h)\frac{(m+1)(-z_0+h)}{\Gamma(2-\alpha)} \left(\frac{2+m}{\alpha} \right)^{-\alpha}  h\left(1-\frac{-z_0+h}{-z_0}\right)^{1-\alpha}\right)^\frac{1}{m+1},
		\end{split}
	\end{equation}
	that is, for $z \rightarrow -z_0^+$, we have
	\begin{equation}
		U(z) \sim \left(\frac{(m+1) z}{\Gamma(2-\alpha)z_0^{1-\alpha}} \left(\frac{\alpha}{2+m} \right)^{\alpha} \right)^\frac{1}{m} (z_0+z)^\frac{2-\alpha}{m},
	\end{equation}
	where in the last asymptotic equality we have used the Lebesgue differentiation theorem, which implies that for any measurable function $f$ we have the following.
	\begin{equation}
		\lim\limits_{h\rightarrow 0^+} \frac{1}{h} \int_a^{a+h} f(x) dx = f(a).
	\end{equation}
	By elementary manipulations and simplification by the common term, we obtain the asymptotic behavior of the solution for $z\rightarrow -z_0^+$. \\
	The last part of the proof involves obtaining a nontrivial bound from below. Bounding the kernel from below, according to Proposition \ref{prop:Kernel}, yields
	\begin{equation}
		U(z)^{m+1} \geq -\underbrace{\frac{\alpha(m+1)}{(2+m)\Gamma(1-\alpha)} \beta\left(1+\frac{\alpha}{2+m},1-\alpha\right)}_{=: D_{\alpha,m}} \int_{-z_0}^z \left(1-\left(\frac{z}{\tau}\right)^\frac{2+m}{\alpha}\right)^{1-\alpha} \tau U(\tau) d\tau.
	\end{equation}
	Now, from the asymptotic behavior of $U$ we know that $U(z) = O((z+z_0)^{(2-\alpha)/m})$ as $z \rightarrow -z_0^+$, hence, it is natural to introduce a new positive function $V=V(z)$ according to
	\begin{equation}\label{eqn:VFunction}
		U(z) = (z+z_0)^\frac{2-\alpha}{m} V(z+z_0).
	\end{equation}
	Then, it satisfies
	\begin{equation}
		(z+z_0)^\frac{(2-\alpha)(m+1)}{m} V(z+z_0)^{m+1} \geq -D_{\alpha,m} \int_{-z_0}^z \left(1-\left(\frac{z}{\tau}\right)^\frac{2+m}{\alpha}\right)^{1-\alpha} \tau (\tau+z_0)^\frac{2-\alpha}{m} V(\tau + z_0) d\tau.
	\end{equation}
	Introduce a new integration variable to make the integration limits constant by $\tau = s z - (1-s)z_0$, then
	\begin{equation}
		\begin{split}
			(z+z_0)^\frac{(2-\alpha)(m+1)}{m}& V(z+z_0)^{m+1} \geq -D_{\alpha,m} (z+z_0)^{1+\frac{2-\alpha}{m}} \times \\
			&\times\int_0^1 \left[\left(1-\left(\frac{z}{s z - (1-s)z_0}\right)^\frac{2+m}{\alpha}\right)^{1-\alpha} (s z - (1-s)z_0)\right] s^\frac{2-\alpha}{m} V(s(z + z_0)) ds.
		\end{split}
	\end{equation}
	Now, the part of the integrand in the square bracket is a $s-$concave function, therefore it is always above a secant spanning from $s = 0$ to $s=1$. Hence,
	\begin{equation}
		\begin{split}
			(z+z_0)^\frac{(2-\alpha)(m+1)}{m} V(z+z_0)^{m+1} &\geq D_{\alpha,m} z_0 (z+z_0)^{1+\frac{2-\alpha}{m}} \left(1-\left(\frac{-z}{z_0}\right)^\frac{2+m}{\alpha}\right)^{1-\alpha} \times \\
			&\times\int_0^1  s^\frac{2-\alpha}{m} V(s(z + z_0)) d\tau.
		\end{split}
	\end{equation}
	Furthermore, we have $(-z/z_0)^{(2+m)/\alpha} \leq -z/z_0$ since $0\leq -z/z_0 \leq 1$ and $(2+m)/\alpha > 1$, which yields
	\begin{equation}
		\begin{split}
			(z+z_0)^\frac{(2-\alpha)(m+1)}{m} V(z+z_0)^{m+1} \geq D_{\alpha,m} z_0^\alpha (z+z_0)^{1+\frac{2-\alpha}{m}} \left(z+z_0\right)^{1-\alpha}\int_0^1  s^\frac{2-\alpha}{m} V(s(z + z_0)) d\tau.
		\end{split}
	\end{equation}
	The $z+z_0$ term can now be canceled, and we arrive at
	\begin{equation}
		V(z+z_0)^{m+1} \geq D_{\alpha,m} z_0^\alpha \int_0^1  s^\frac{2-\alpha}{m} V(s(z + z_0)) d\tau.
	\end{equation}
	We can now use Lemma \ref{lem:Gronwall} with $y(t) := V(t)$ and $\gamma := (2-\alpha)/m$ to obtain that
	\begin{equation}
		V(z+z_0)\geq \left(\frac{D_{\alpha,m}z_0^\alpha}{1+\frac{2-\alpha}{m}}\right)^\frac{1}{m}.
	\end{equation}
	Going back to the definition of $V$ as in \eqref{eqn:VFunction} concludes the proof. 
\end{proof}

\begin{rem}
	Theorem \ref{thm:Solution} states that we know that the solution is H\"older continuous. The \textit{a-priori} found exponent $(1-\alpha)/(1+m)$ is probably not optimal, since the asymptotic behavior suggests that $U$ should be $(2-\alpha)/m-$ H\"older continuous. Of course, these results are consistent because $(1-\alpha)/(1+m) < \min\{1,(2-\alpha)/m\}$. Moreover, if the solution were that regular, it would globally be Lipschitz continuous for $m\leq 2-\alpha$.
\end{rem}

\begin{rem}\label{rem:Symmetry}
	The derivative at $z=0$ of the solution of \eqref{eqn:FixedPoint2} satisfies a simple relation \eqref{eqn:Derivative0}. In particular, it is positive for $0\leq \alpha<1$ and converges to zero for $\alpha \rightarrow 1^-$ (since $U(0)$ is bounded then). Recalling that $U$ represents only the left-half of the symmetric Barenblatt solution, we conclude that it is smooth at the origin only in the classical case. This is in line with the findings for the fundamental solution of the linear subdiffusion equation \cite{mainardi2001fundamental}. 
\end{rem}


\subsection{Application to the porous medium equation}

Using the appropriate bounds for the kernel, we can prove that there exists a unique $z_0$ such that the mass conservation \eqref{eqn:MassConservationHalf} is satisfied.
\begin{prop}\label{prop:Mass}
	Let $U=U(z;z_0)$ be the solution of \eqref{eqn:FixedPoint2} for a given $z_0 > 0$. Then, there exists a unique $z_0^*> 0$ such that \eqref{eqn:MassConservationHalf} is satisfied for $z_0 = z_0^*$. 
\end{prop}
\begin{proof}
	We will use the Darboux principle to show that the increasing function 
	\begin{equation}
		M(z_0) := \int_{-z_0}^0 U(z; z_0) dz,
	\end{equation}
	attains a value $1/2$ for some unique $z_0 = z_0^*$. First, by Theorem \ref{thm:Solution} we know that $U$ is bounded. Therefore, on the one hand for some constant $D^{(1)}_{\alpha,m}$, we have
	\begin{equation}
		M(z_0) \leq D^{(1)}_{\alpha, m} z_0^\frac{2}{m} \int_{-z_0}^0 dz = D^{(1)}_{\alpha, m} z_0^{1+\frac{2}{m}}.
	\end{equation}
	Hence, $M(z_0) \rightarrow 0$ when $z_0 \rightarrow 0^+$. On the other hand, using the lower bound with another constant $D^{(2)}_{\alpha,m}$
	\begin{equation}
		M(z_0) \geq D^{(2)}_{\alpha,m} z_0^{\frac{\alpha}{m}}\int_{-z_0}^0 (z+z_0)^\frac{2-\alpha}{m} dz = \frac{D^{(2)}_{\alpha,m} }{1+\frac{2-\alpha}{m}} z_0^{1+\frac{2}{m}},
	\end{equation}
	which forces that $M(z_0) \rightarrow \infty$ for $z_0 \rightarrow \infty$. Therefore, $M: [0,\infty) \mapsto [0, \infty)$ is a continuous function that reaches all values in its range. Therefore, from the Darboux principle there exists $z_0^*$ such that $M(z_0^*) = 1/2$. The uniqueness of $z_0^*$ follows from monotonicity. The proof is complete. 
\end{proof}

\begin{rem}
	The above proposition states a somewhat stronger result: $M(z_0) = O(z_0^{1+2/m})$ as $z_0\rightarrow\infty$. That is, the growth of the mass is independent of $\alpha$ and we know its precise rate.
\end{rem}

By all of the above calculations, we have proved the following result. 
\begin{thm}\label{eqn:Barenblatt}
	Let $U=U(z;z_0)$  be a nontrivial solution to \eqref{eqn:FixedPoint2} and let $t_0>0$ be arbitrary. For $z_0^*$ from Proposition \ref{prop:Mass} define 
	\begin{equation}\label{eqn:BarenblattSolution}
		u(x,t) = t^{-\frac{\alpha}{2+m}}
		\begin{cases}
			U(-|x| t^{-\frac{\alpha}{2+m}}; z_0^*), & |x| \leq z_0^* t^{\frac{\alpha}{2+m}}, \\
			0, & \text{otherwise}, \\
		\end{cases}
		\quad 
		x \in \mathbb{R}, \quad t \in [t_0, \infty),
	\end{equation}
	Then, for arbitrary $0<z_1<z_2<z_0^*$ with $0 < z_1 \leq |x| t^{-\alpha/(2+m)} \leq z_2 < z_0^*$ the function \eqref{eqn:BarenblattSolution} is the strong solution of \eqref{eqn:MainEq0} and satisfies the mass conservation condition \eqref{eqn:MassConservation} for all $x\in\mathbb{R}$.
\end{thm}
\begin{proof}
	Taking into account all the above reasoning, we have to prove that \eqref{eqn:BarenblattSolution} is the strong solution of \eqref{eqn:MainEq0}. Note that by Theorem \ref{thm:Solution} for any $x$ and $t$ in our region, the solution $u(x,t)$ is twice differentiable in space and at least once in time. Therefore, we can differentiate \eqref{eqn:FixedPoint2} and reverse the steps that lead to it from \eqref{eqn:MainEq0}. Hence, \eqref{eqn:BarenblattSolution} is the pointwise (strong) solution of our main equation. 
\end{proof}

\begin{rem}
	The above theorem states the result about a strong solution only in a portion of our domain $\mathbb{R} \times (0, \infty)$. As was noted in the Introduction, the full account of the theory should involve a suitable concept of weak solutions and uniqueness. Due to its complexity and participation, we leave this topic for future work. 
\end{rem}

\section{Numerical scheme}
In this section, we design a numerical scheme for obtaining Barenblatt solutions of \eqref{eqn:MainEq0}. According to Theorem \ref{thm:Solution} we need only to solve \eqref{eqn:FixedPoint2} for a chosen $z_0 = z_0^*$ such that the mass conservation \eqref{eqn:MassConservationHalf} is satisfied.

\subsection{Analysis of the scheme}
Introduce the uniform grid $z_n := -z_0 + n h$ for $n\in\mathbb{N}$ and $h = z_0/N$ with $N$ being the total number of interval divisions. Choose an open quadrature to discretize the integral in \eqref{eqn:FixedPoint2}, that is
\begin{equation}\label{eqn:Quadrature}
	\frac{m+1}{\Gamma(1-\alpha)}\int_{-z_0}^{z_n} K(z_n, \tau) U(\tau) d\tau = h \sum_{i=1}^{n-1} w_{n,i} U(z_i) + \delta_n(h),
\end{equation}
where $\delta_n(h)$ is the local discretization error and $w_{n,i}$ are weights of the quadrature. For example, using the product rectangle rule for a $\gamma$-H\"older function, we have
\begin{equation}\label{eqn:WeightsRectangle}
	w_{n,i} = \int_{z_{i}}^{z_{i+1}} K(z_n, \tau) d\tau, \quad \delta_n(h) = L(u) h^\gamma \left(\int_{-z_0}^{z_n} K(z_n, \tau) d\tau\right),
\end{equation}
where $L(u)$ is the H\"older constant of $u$. Choosing different quadratures, we can generate various other schemes (see, for example, \cite{okrasinska2022second}). Moreover, notice that we choose to approximate the solution by the rectangle rule leaving the kernel intact (hence, the 'product rule'). This improves accuracy in the weak-regularity setting. We note that since our solution is only H\"older continuous, we cannot hope that using the high-order scheme will provide more accurate approximations. It seems that the rectangle rule is a reasonable choice that balances the small amount of computational cost, ease of implementation, and sufficient accuracy given the weak regularity of the solution. We elaborate on this topic numerically in the next section. 

Therefore, if $U_n$ is the numerical approximation to the solution of \eqref{eqn:FixedPoint2}, we can devise a simple explicit scheme by truncating $\delta_n(h)$
\begin{equation}\label{eqn:NumericalScheme}
	U_n = \left(h \sum_{i=1}^{n-1} w_{n,i} U_i\right)^\frac{1}{m+1}.
\end{equation}
We also defined the global discretization error
\begin{equation}\label{eqn:GlobalDiscretisationError}
	\delta(h) := \max_n |\delta_n(h)|.
\end{equation}
In reality, the numerical approximation is a function of $z_0$, that is, $U_n = U_n(z_0)$, but for clarity of presentation, we omit this from the notation where justified. In the same way as in the continuous integral equation, the discrete one has a trivial solution $U_n \equiv 0$. To obtain a nontrivial one, we have to specify a positive starting value $U_1$ (since $U_0 = 0$). Then, due to monotonicity, the next values will be positive. Probably the most natural choice for the starting value is the \textit{exact} asymptotical value according to Theorem \ref{thm:Solution}, that is
\begin{equation}\label{eqn:InitialCondition}
	U_0 = 0, \quad U_1 = \left(\frac{\frac{(m+1) z_0^\alpha}{\Gamma(2-\alpha)} \left(\frac{\alpha}{2+m} \right)^{\alpha}}{1+\frac{2-\alpha}{m}}\right)^\frac{1}{m} h^\frac{2-\alpha}{m}.
\end{equation}
In this way, we introduce the discretization error of order $o(h^{(2-\alpha)/m})$ as $h\rightarrow 0$.

The integral condition \eqref{eqn:MassConservationHalf} is treated in the following way. Given the numerical approximation $\left\{U_n\right\}_n$ define the expected tolerance $TOL > 0$ and the function 
\begin{equation}\label{eqn:MassConservationDiscrete}
	F(z_0) = h \sum_{n = 1}^N v_{N,n} U_n - \frac{1}{2},
\end{equation}
where $v_{N,n}$ are the weights of a chosen quadrature (for example, trapezoid or Gauss). Starting with some given initial value $z_0 = z_0^{(0)}$, we apply a root-finding algorithm such as Newton or Brent. After acquiring the desired tolerance, we set $z_0 = z_0^*$ when $|F(z_0)| < TOL$. The corresponding solution of the integral equation \eqref{eqn:FixedPoint2} is the desired one. The overall procedure for obtaining a numerical approximation to our problem is presented in Alg. \ref{alg:NumericalScheme}. 

\begin{algorithm}
	\caption{Numerical solution of \eqref{eqn:FixedPoint2} with \eqref{eqn:MassConservationHalf}. }
	\label{alg:NumericalScheme}
	\begin{algorithmic}[1]
		\Require $z_0^{(0)} > 0$, $N > 0$, tolerance $TOL$
		\Ensure $z_0$ such that $|F(z_0)| < TOL$ where $F$ is defined in (\ref{eqn:MassConservationDiscrete})
		\State Compute weights $w_{n,i}$ from \eqref{eqn:Quadrature} and $v_{N,n}$ from \eqref{eqn:MassConservationDiscrete}
		\State Set $U_1$ according to equation \eqref{eqn:InitialCondition}
		\While{$|F(z_0)| \geq TOL$}
		\State Set step size $h \gets z_0 / N$
		\State Compute $U_n$ using \eqref{eqn:NumericalScheme} for $2\leq n \leq N$
		\State Update $z_0$ using the root-finding algorithm for $F$
		\EndWhile
		\State\Return $z_0$, $\left\{U_n(z_0)\right\}_n$
	\end{algorithmic}
\end{algorithm}

The main result of this section is the convergence proof of our scheme \eqref{eqn:NumericalScheme}.
\begin{thm}\label{thm:Covergence}
	Let $U=U(z)$ be the solution of \eqref{eqn:FixedPoint2}, while $U_n$ its numerical approximation at $z_n$ is calculated from \eqref{eqn:NumericalScheme}. Assume that the order of discretization of the starting step \eqref{eqn:InitialCondition} is greater than or equal to the order of the global discretization error $\delta(h)$ defined in \eqref{eqn:GlobalDiscretisationError}. Then, 
	\begin{equation}
		\max_n |U(z_n) - U_n| = O\left(\left(\ln \ln \frac{1}{h}\right) \left(\ln \frac{1}{h}\right)\delta(h) \right), \quad h\rightarrow 0^+.
	\end{equation} 
\end{thm}
\begin{proof}
	We begin by perturbing our equation and the scheme with, yet to be specified, small number $\epsilon=\epsilon(h)$ which vanishes as $h\rightarrow 0^+$
	\begin{equation}
		U(z) +\epsilon(h) = \left(\frac{m+1}{\Gamma(1-\alpha)}\int_{-z_0}^z K(z,\tau)U(\tau) d\tau\right)^\frac{1}{m+1}, \quad U_n + \epsilon(h) = \left(h \sum_{i=1}^{n-1} w_{n,i} U_i\right)^\frac{1}{m+1}.
	\end{equation} 
	The solutions of the above converge to the solution of unperturbed equations with $\epsilon(h) \rightarrow 0$. Define the error of the scheme by $e_n := U(z_n) - U_n$. Then, by taking the power and subtracting we obtain
	\begin{equation}
		\begin{split}
			(U(z) +\epsilon(h))^{m+1} - (U_n + \epsilon(h))^{m+1} 
			&= \frac{m+1}{\Gamma(1-\alpha)}\int_{-z_0}^z K(z,\tau)U(\tau) d\tau - h \sum_{i=1}^{n-1} w_{n,i} U_i \\
			&= h \sum_{i=1}^{n-1} w_{n,i} e_i + \delta_n(h),
		\end{split}
	\end{equation}
	where in the second equality we have used the definition of the discretization error \eqref{eqn:Quadrature}. Now, by the mean value theorem for any positive numbers $a$ and $b$ we have $|a^{m+1}-b^{m+1}| = (m+1) c^m |a-b|$ with $c$ between $a$ and $b$. Applying this to the above we obtain
	\begin{equation}
		(m+1) \xi^{m} |e_n| \leq h \sum_{i=1}^{n-1} |w_{n,i}| |e_i| + \delta(h),
	\end{equation}
	with where the discretization error was bounded by its global value. Here, $\xi$ lies between $U(z) +\epsilon(h)$ and $U_n + \epsilon(h)$, and hence we always have $\xi \geq \epsilon(h)$. Therefore,
	\begin{equation}
		|e_n| \leq \frac{\epsilon^{-m}}{m+1} h \sum_{i=1}^{n-1} |w_{n,i}| |e_i| + \frac{\epsilon^{-m}}{m+1}\delta(h).
	\end{equation}
	Now, since both the kernel $K$ and the solution $U$ are bounded, there exists a constant $C>0$ such that
	\begin{equation}
		|e_n| \leq C \epsilon^{-m} \left(h \sum_{i=1}^{n-1} |e_i| + \delta(h)\right).
	\end{equation}
	This brings us to the position of using the classical version of the discrete Gronwall inequality (see, for example, \cite{ames1997inequalities}) which, under our assumption that $|U(z_1) - U_1| = |e_1| = O(\delta(h))$ as $h\rightarrow 0$, yields
	\begin{equation}
		|e_n| \leq C \epsilon^{-m} \delta(h) e^{C n h \epsilon^{-m}} \leq C \epsilon^{-m} \delta(h) e^{C z_0 \epsilon^{-m}},
	\end{equation}
	where we used the fact that $n h = z_n + z_0 \leq z_0$. Now, we have to choose the precise form of the perturbation $\epsilon(h)$. There are infinitely many such choices that preserve the order of $\delta(h)$ (up to slowly varying factors) while still converging to zero, for example
	\begin{equation}
		\epsilon(h) = \left(\frac{1}{C z_0} \ln \ln \frac{1}{h}\right)^{-\frac{1}{m}}.
	\end{equation}
	Then,
	\begin{equation}
		|e_n| \leq \frac{1}{z_0} \left(\ln \ln \frac{1}{h}\right) \left(\ln \frac{1}{h}\right) \delta(h),
	\end{equation}
	which ends the proof by taking the maximum over $n$. 
\end{proof}

\begin{rem}
	Arbitrarily chosen logarithmic factors in the above result are not relevant in practical applications because of their very slow growth. Hence, we can conclude that the above theorem states that the scheme converges \textit{essentially} with the same order as the discretization error $\delta(h)$.
\end{rem}

\subsection{Numerical illustration}
We illustrate our numerical scheme with some examples. First, in Fig. \ref{fig:Barenblatt} we can see several solutions for fixed $m=1$ and varying $\alpha$. For our purposes, we use the rectangle rule \eqref{eqn:WeightsRectangle}. The half-size of the compact support $z_0$ was found with the bisection method with the trapezoidal rule for the quadrature of $U$. On average, the algorithm converged with a relative tolerance of $10^{-4}$ in less than 10 iterations. Note that due to only H\"older regularity of the solution, it is not advised to use higher-order methods such as Newton's. 

In the plot in Fig. \ref{fig:Barenblatt} we have depicted the even solution reflected through the origin for a clearer presentation. As we can see, the classical solution for $\alpha \rightarrow 1^-$ is a smooth function with vanishing derivative at $z = 0$. This corresponds to the exact formula \eqref{eqn:BarenblattExact} very accurately with a maximum error of $3\times 10^{-3}$. The situation is dramatically different for $0<\alpha <1$, where we observe a sharp cusp and a lack of differentiability. This was already observed in \cite{mainardi2001fundamental} for the linear subdiffusion equation and its fundamental solution (see also Remark \ref{rem:Symmetry}). 

\begin{figure}
	\centering
	\includegraphics[scale = 0.8]{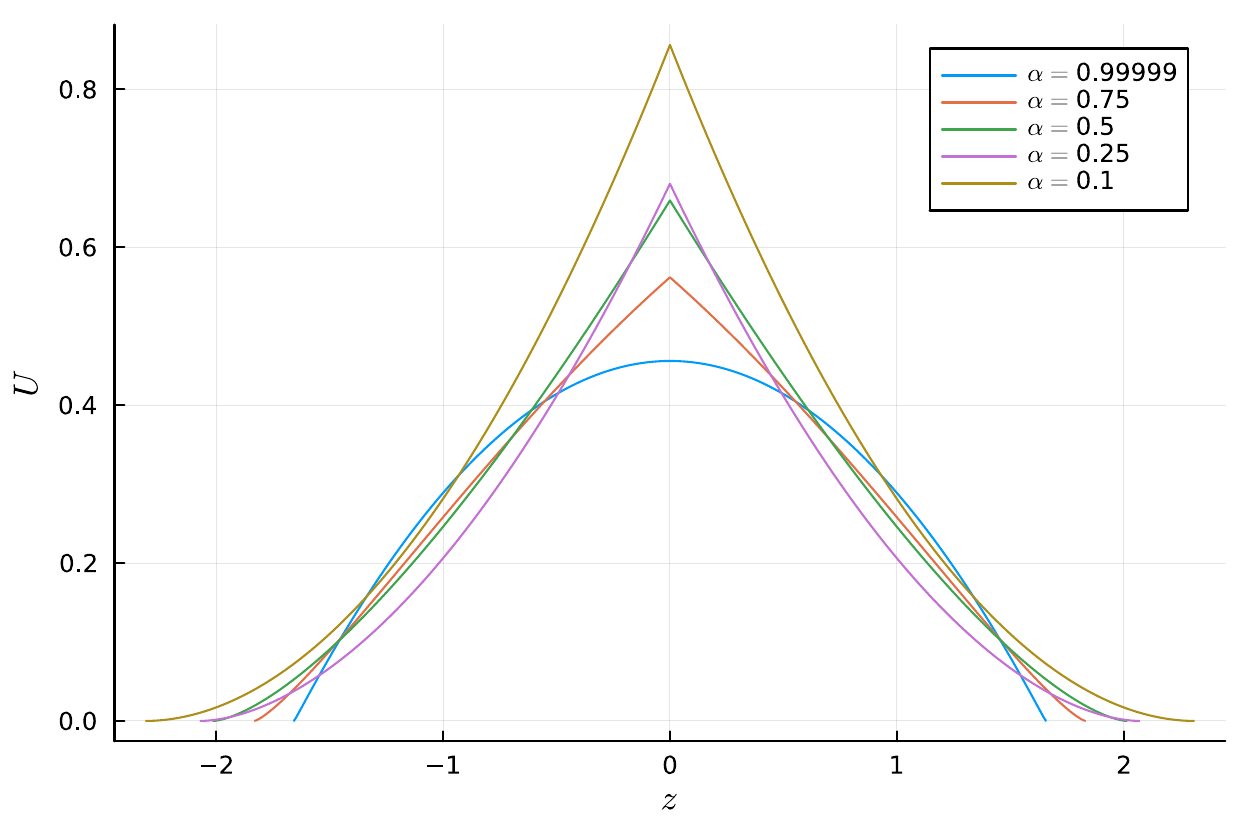}
	\caption{Solutions of \eqref{eqn:NumericalScheme} obtained with rectangle weights \eqref{eqn:WeightsRectangle} according the the Alg. \ref{alg:NumericalScheme}. Here, $m = 1$ and number of steps $N = 2^{10}$. }
	\label{fig:Barenblatt} 
\end{figure}

The dependence on $z_0$ of the solution $U=U(z; z_0)$ is depicted in Fig. \ref{fig:Z0Dependence}. It is clearly seen that $U(z, \cdot)$ increases with increasing $z_0$ (decreasing $-z_0$). Moreover, as can also be inferred from the graphs, the overall shape of the function resembles $C (|z|+z_0)^{(2-\alpha)/m}$ as was also suggested by the asymptotic behavior and lower bound found in Theorem \ref{thm:Solution}. For example, for fixed $m = 1.5$ we obtain $(2-\alpha)/m = 1.1(6) > 1$ (convex) for $\alpha = 0.25$ and $(2-\alpha)/m = 0.8(3)$ (concave) for $\alpha = 0.75$. 

\begin{figure}
	\centering
	\includegraphics[scale = 0.42]{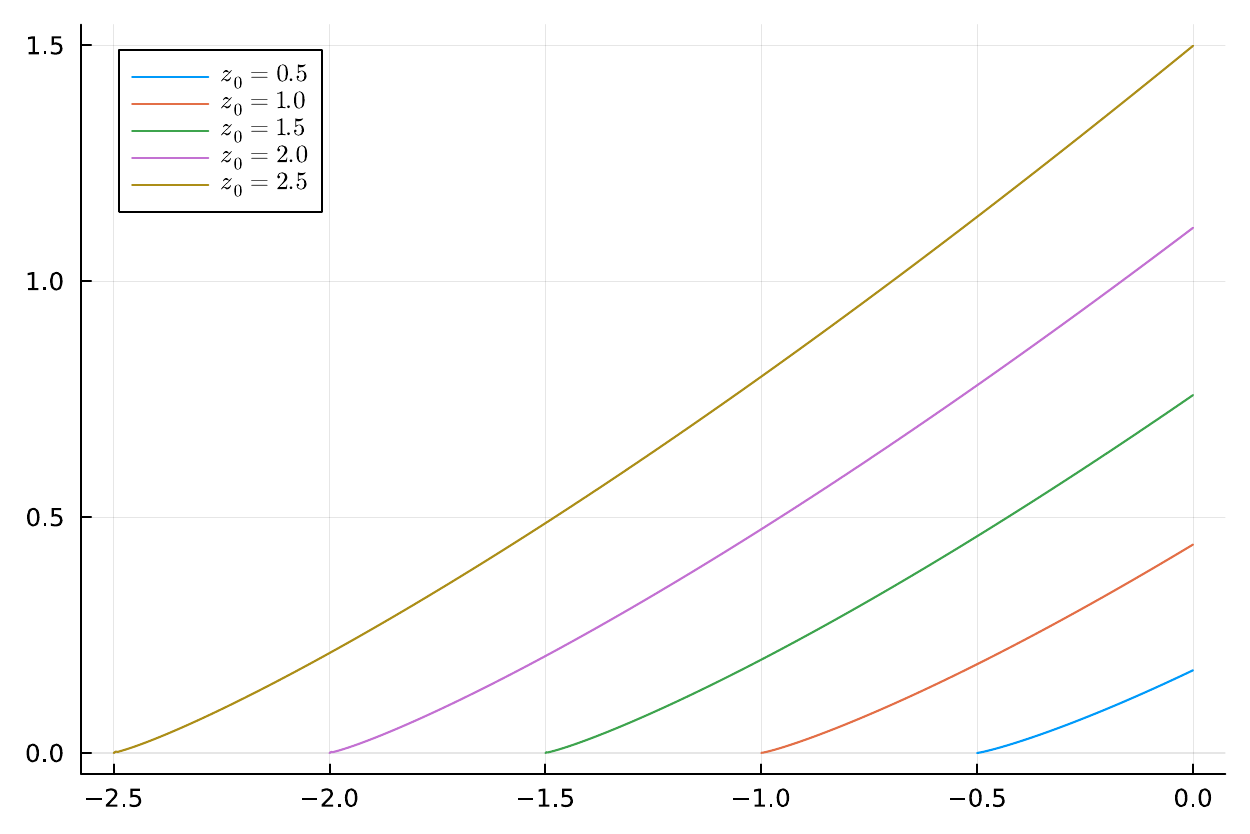}
	\includegraphics[scale = 0.42]{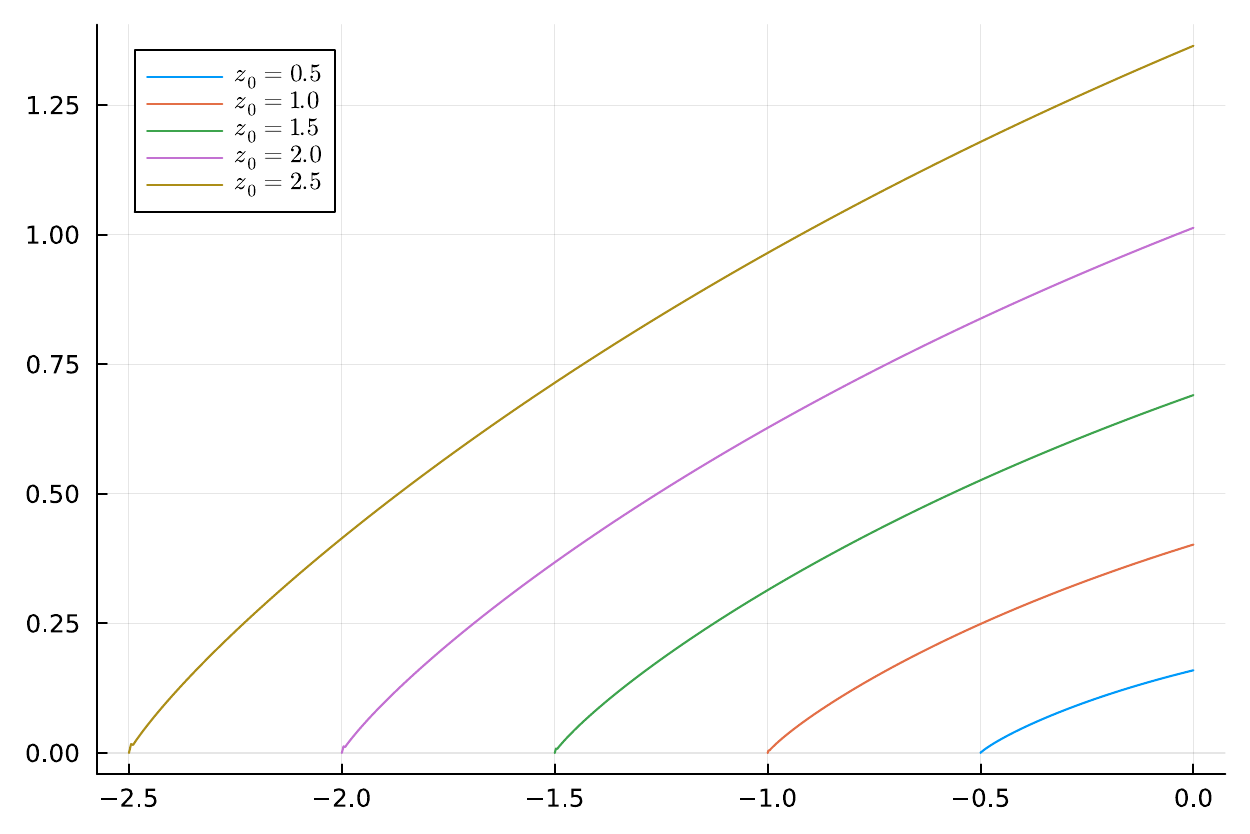}
	\caption{Solutions of \eqref{eqn:NumericalScheme} obtained with rectangle weights \eqref{eqn:WeightsRectangle} for difference $z_0>0$ and $\alpha = 0.25$ (left), $\alpha = 0.75$ (right). Here, $m = 1.5$ and number of steps $N = 2^{10}$. }
	\label{fig:Z0Dependence} 
\end{figure}

As a next test of our scheme, we compute the half-size of the compact support, that is $z_0$ for the mass conservation \eqref{eqn:MassConservationHalf} to be satisfied. We also compare it with the value for the classical case where the formula is available \eqref{eqn:BarenblattExact}. Results are presented in Tab. \ref{tab:CompactSupport}. Note that agreement with the theoretical values is high taking into account the lower regularity of the solution that influences the accuracy of both the main scheme \eqref{eqn:NumericalScheme} and the quadrature used to calculate its integral. We can conclude that regardless of these difficulties, our scheme is robust and can also tackle the classical case. 

\begin{table}
	\centering
	\begin{tabular}{cccccc}
		\toprule
		m & 1 & 3 & 5 & 7 & 9 \\
		\midrule
		$z_0$ exact & 1.650 & 1.159 & 0.961 & 0.856 &  0.791 \\
		$z_0$ num. & 1.654 & 1.101 & 0.961 & 0.769 & 0.740 \\
		\bottomrule
	\end{tabular}
	\caption{Theoretical values of $z_0$ obtained from the exact formula \eqref{eqn:BarenblattExact} for $\alpha = 1$ and its numerical approximations for various $m$. }
	\label{tab:CompactSupport}
\end{table}

The next test is the estimation of the order of convergence. Since we want to compute the error in the case $0<\alpha<1$, when the exact formulas are not available, we have to use an indirect method. One of such is to use an idea based on extrapolation. Let $U^{(N)}_n$ be the solution computed with \eqref{eqn:NumericalScheme} with the number of subdivisions equal to $N$. If it is large enough, we can expect that the proxy for the error estimate is
\begin{equation}
	\max_n|U^{(N)}_n - U_n^{(2N)}| \approx C h^p,
\end{equation}
where $p>0$ is the estimated order of convergence and the step is $h=z_0/N$. If we do the same computation for a doubled number of subdivisions, we obtain
\begin{equation}
	\max_n|U^{(2N)}_n - U_n^{(4N)}| \approx C \left(\frac{h}{2}\right)^p,
\end{equation}
where we assume that the error constant $C>0$ does not change substantially. Therefore, dividing the two above formulas and taking them, we obtain an estimate of the order
\begin{equation}
	p \approx \log_2 \frac{\max_n|U^{(N)}_n - U_n^{(2N)}|}{\max_n|U^{(2N)}_n - U_n^{(4N)}|}.
\end{equation}
This method allows us to find an approximate value of the order of convergence without a great computational cost. Another frequently used procedure is to find $U^{(N)}$ with a very large $N$ and treat it as an estimate of the exact solution (since we know the scheme converges). In our case, however, due to low regularity, this can be prohibitively expensive, and we use the extrapolation method.

Results of order estimation are presented in Tab. \ref{tab:Order} for various $\alpha$ and $m$. The base for the extrapolation was chosen as $N = 2^{11}$, which is a relatively large number. As can be immediately seen, the estimated order is slightly lower than $1$ and fairly independent of $\alpha$ which is a good feature since in many numerical schemes the global order deteriorates as $\alpha \rightarrow 0^+$ as the solution becomes less and less regular. This $\alpha$-robustness can be explained by the fact that our starting condition \eqref{eqn:InitialCondition} acts as a correction that incorporates the power-type nature of the solution. For higher $m$, the order approaches $1$. We know from Theorem \ref{thm:Solution} that our solution is actually smooth away from $z = -z_0$ and H\"older at that point. Therefore, we expect a higher order of convergence than predicted by \eqref{eqn:WeightsRectangle} which assumes global H\"older continuity. For smooth functions, this is exactly equal to $1$. However, the fact that our results show that the order is less than that may be explained by the logarithmic corrections in Theorem \ref{thm:Covergence} and that the extrapolation method is, in essence, only an approximation of the factual order. Nevertheless, these computations verify that the numerical scheme \eqref{eqn:NumericalScheme} is a robust and relatively accurate method for finding Bareblatt's solutions of the time-fractional porous medium equation. 

\begin{table}
	\centering
	\begin{tabular}{cccccc}
		\toprule
		$\alpha\backslash m$ & 1 & 3 & 5 & 7 & 9 \\
		\midrule
		0.999 & 0.69 & 0.73 & 0.74 & 0.75 & 0.75 \\
		0.9 & 0.71 & 0.82 & 0.87 & 0.9 & 0.92 \\
		0.5 & 0.74 & 0.86 & 0.9 & 0.92 & 0.94 \\ 
		0.2 & 0.75 & 0.87 & 0.9 & 0.92 & 0.93 \\ 
		0.01 & 0.75 & 0.86 & 0.9 & 0.92 & 0.93 \\
		\bottomrule
	\end{tabular}
	\caption{Estimate order of convergence of the numerical scheme \eqref{eqn:NumericalScheme} for various $\alpha$ (rows) and $m$ (columns). }
	\label{tab:Order}
\end{table}

\section{Conclusion}
We have found a Barenblatt solution of the time-fractional PME via the integral equation approach arising in the search for similarity solutions. This methodology enables us to derive several quantitative and qualitative properties of the solution, such as monotonicity, estimates, regularity, and mass conservation. Moreover, we devised a convergent numerical scheme to approximate the Barenblatt solution.

As we mentioned above, future work will concern proving that Barenblatt solution we have found using the integral equation approach is indeed a unique weak solution of \eqref{eqn:MainEq0} with measure initial data. In addition, an interesting direction of research is to generalize our results to higher spatial dimensions to provide a full account of the situation. However, this would require us to consider a different class of integral equations, and thus we leave it for our next topic. 

\section*{Acknowledgement}
\L.P. has been supported by the National Science Centre, Poland (NCN) under the grant Sonata Bis with a number NCN 2020/38/E/ST1/00153. J.C. and K. S. are partially supported by the project PID2023-148028NB-I00.


\end{document}